\title{Enumeration and randomized constructions of hypertrees}
\author{Nati Linial\thanks{Department of Computer Science, Hebrew University, Jerusalem 91904,
    Israel. e-mail: nati@cs.huji.ac.il~. Supported by ERC grant 339096 "High-dimensional combinatorics".} \and  {Yuval Peled\thanks{Department of Computer Science, Hebrew University, Jerusalem 91904,
    Israel. e-mail: yuvalp@cs.huji.ac.il~. YP is grateful to the Azrieli foundation for the award of an Azrieli Fellowship.}
}}

\date{\today}
\documentclass[11pt]{article}
\usepackage{amssymb,fullpage}
\usepackage{xspace}
\usepackage{latexsym}
\usepackage{times}
\usepackage{amsfonts}
\usepackage{amsmath}

\usepackage{mathrsfs}
\usepackage{bbm}

\usepackage{verbatim}

\usepackage{amsthm}
\usepackage{amsmath}
\usepackage{amssymb}
\usepackage{graphicx}
\usepackage{subfig}
\usepackage{epstopdf}
\usepackage{titling}

\usepackage{enumerate}

\usepackage{algpseudocode}
\usepackage{algorithm}
\usepackage{algorithmicx}
\usepackage{tikz}
\usetikzlibrary{shapes.gates.logic.US,trees,positioning,arrows}

\newcommand{\ignore}[1]{}

\newcommand{\R}{\ensuremath{\mathbb R}}
\newcommand{\Z}{\ensuremath{\mathbb Z}}
\newcommand{\Q}{\ensuremath{\mathbb Q}}
\newcommand{\N}{\ensuremath{\mathbb N}}
\newcommand{\E}{\ensuremath{\mathbb E}}

\newcommand{\F}{\ensuremath{\mathbb F}}
\newcommand{\Tnd}{\ensuremath{\mathcal T_{n,d}}}
\newcommand{\Cnd}{\ensuremath{\mathcal C_{n,d}}}
\newcommand{\Cndvar}[2]{\ensuremath{\mathcal C_{#1,#2}}}

\newtheorem{theorem}{Theorem}[section]

\newtheorem{lemma}[theorem]{Lemma}
\newtheorem{claim}[theorem]{Claim}
\newtheorem{corollary}[theorem]{Corollary}

\newtheorem{fact}[theorem]{Fact}

\newtheorem{conjecture}[theorem]{Conjecture}
\newtheorem{definition}[theorem]{Definition}
\newtheorem{question}[theorem]{Question}
\newtheorem{remark}[theorem]{Remark}

\def \dim {{\rm dim}}
\def \im {{\rm Im}}
\def \ker {{\rm ker}}
\def \rank {{\rm rank}}

\def \LMc {Y_d\left(n,\frac cn\right)}

\def \SH {\mbox{SH}}

\begin{document}
\maketitle
\begin{abstract}
Over thirty years ago, Kalai proved a beautiful $d$-dimensional analog of Cayley's formula for the number of $n$-vertex trees. He enumerated $d$-dimensional hypertrees weighted by the squared size of their $(d-1)$-dimensional homology group. This, however, does not answer the more basic problem of unweighted enumeration of $d$-hypertrees, which is our concern here. Our main result, Theorem \ref{thm:1}, significantly improves the lower bound for the number of $d$-hypertrees. In addition, we study a random $1$-out model of $d$-complexes where every $(d-1)$-dimensional face selects a random $d$-face containing it, and show it has a negligible $d$-dimensional homology.
\end{abstract}
\section{Introduction}


Trees are among the most fundamental objects in discrete mathematics and computer science, as documented in innumerable theoretical and applied studies. As part of our ongoing research in {\em high-dimensional combinatorics}, we study here {\em high-dimensional trees}.
In graph theory, a {\em tree} is characterized by being {\em connected} and {\em acyclic}. Since both these properties are topological in nature, it makes sense to consider them in higher dimensional simplicial complexes as well. This was indeed done over thirty years ago in a beautiful paper by Kalai \cite{kalai}.

From a topological perspective, a graph is a $1$-dimensional simplicial complex. Also, connectivity and cycles in graphs are expressible in the language of {\em simplicial homology}. Namely, connectivity is the vanishing of the zeroth homology and cycles are elements of the graph's first homology. Kalai's definition applies the same line of thought to $d$-dimensional complexes regarding the $(d-1)$-st and $d$-th homology:
\begin{definition}
A {\em $d$-hypertree} $T$ is a $d$-dimensional simplicial complex with a full $(d-1)$-skeleton such that both $H_{d-1}(T;\Q)=0$ and $H_{d}(T;\Q)=0$.
\end{definition}
Recall that a $d$-dimensional simplicial complex ($d$-complex in short) has a full $(d-1)$-skeleton if it contains all the faces of dimension less than $d$ that are spanned by its vertex set. In this paper, unless stated otherwise, all $d$-complexes are assumed to have a full $(d-1)$-skeleton, and we sometimes identify a $d$-complex with its set of $d$-dimensional faces. Note also that it makes sense to consider a similar notion of a hypertree where $\Q$ is replaced by a different commutative ring of coefficients. However, this would yield a different class of complexes, and unless stated otherwise, we stick to $\Q$-acyclic complexes i.e., $d$-hypertrees over $\Q$.

The notion of a $d$-hypertree is expressible as well in terms of elementary linear algebra, and specifically the {\em boundary operator} (or matrix) $\partial_d$ that maps a $d$-face to the linear sum of its $(d-1)$-subfaces. A $d$-hypertree is a set of $d$-faces whose corresponding columns in $\partial_d$ form a basis for the column space of $\partial_d$.
For example, the $d$-dimensional star, which is comprised of the $d$-faces that contain a specific vertex, is a $d$-hypertree. Clearly, an $n$-vertex $d$-hypertree has exactly $\binom{n-1}d$ $d$-dimensional faces.

One of the earliest nontrivial discoveries about trees is Cayley's formula, which states that the number of trees on $n$ labeled vertices is $n^{n-2}$. Kalai found a beautiful generalization of this formula in higher dimensions. Let $\Tnd$ denote the family of $n$-vertex $d$-hypertrees.
\begin{theorem}
\label{thm:kalai}
Let $d<n$ be integers. Then,
\[
\sum_{T\in\mathcal T_{n,d}} |H_{d-1}(T;\Z)|^2 = n^{\binom{n-2}d}.
\]
\end{theorem}
Recall that $H_{d-1}(T;\Z)$, the {\em torsion} of $T$, is a finite group for every $d$-hypertree $T$. For $d=1$, Kalai's formula reduces to Cayley's formula, since the integral $0$-th homology is always torsion-free. In fact, Kalai's argument is a high-dimensional extension of the Matrix-Tree Theorem.

In the one-dimensional case of graphs, trees can also be defined by the combinatorial notion of collapsibility. An elementary $1$-collapse in a graph is the removal of a leaf (= a vertex of degree one) and the unique edge that contains it. An $n$-vertex graph with $n-1$ edges is a tree if and only if it is collapsible, i.e., it can be reduced to a single vertex by a sequence of elementary $1$-collapses. 

This definition of trees generalizes to a notion of {\em $d$-collapsible $d$-hypertrees}. A $(d-1)$-face $\tau$ in a simplicial complex $X$ is said to be {\em exposed} if there is exactly one $d$-face $\sigma$ in $X$ that contains it. In the {\em elementary $d$-collapse} on $\tau$, we remove $\tau$ and $\sigma$ from $X$. Note that the remaining complex is homotopy equivalent to $X$. We say that $X$ is {\em $d$-collapsible} if it is possible to eliminate all its $d$-faces by a series of elementary collapses. In particular, $d$-collapsibility implies the vanishing of the $d$-th homology. Therefore, a $d$-collapsible $d$-complex with $\binom{n-1}d$ $d$-faces is a $d$-hypertree. Such a complex is called a {\em $d$-collapsible $d$-hypertree}. 

While all $1$-dimensional trees are collapsible, it is conjectured that in higher dimension the situation is entirely different. We denote by $\Cnd$ the set of all $n$-vertex $d$-collapsible hypetrees.
\begin{conjecture}
\label{conj:main}
For every $d\ge 2$, asymptotically almost none of the $d$-hypertrees are $d$-collapsible. Namely, \mbox{$|\Cnd|/|\Tnd| \to 0$}, as $n\to\infty$.
\end{conjecture}

Kalai's formula has motivated several other results of torsion-related weighted enumeration of hypertrees \cite{Duv,Adin,Lyons}. But, despite its remarkable beauty, the formula leaves a substantial gap regarding the question of {\em unweighted enumeration of (labeled) $d$-hypertrees}. Here are the bounds that are mentioned in \cite{kalai}:
\[
 \left( \frac {n}{d+1}\right) ^ {\binom {n-1}d}<|\Tnd| < \left( \frac {e}{d+1}\cdot n\right) ^ {\binom {n-1}d}.
\]
As mentioned, the number of $d$-faces in a $d$-hypertree is $\binom {n-1}d$, and the upper bound only considers the number of ways to  select them from among the total of $\binom{n}{d+1}$. The lower bound follows from the identity in Theorem \ref{thm:kalai} and an upper bound of the size of the torsion of a $d$-hypertree. Our analysis in Section \ref{sec:bounds} yields an elementary proof of this lower bound. 

In fact, the following simple inductive construction, suggested to us by Gil Kalai, yields a better lower bound even for collapsible hypertrees. Let $S$ be a $(d-1)$-collapsible hypertree with vertex set $[n]$, and let $vS=\{v\tau~|~\tau\in S\}$ be a simplicial cone over $S$, where $v$ is a new vertex. Let $T$ be a $d$-collapsible hypertree on $[n]$. The union $T\cup(vS)$ is a $d$-collapsible hypertree on $[n]\cup\{v\}$. Indeed, extend first the $(d-1)$-collapse of $S$ to the cone $vS$ and then collapse $T$.

Consequently, $|\Cndvar{n+1}{d}|\ge|\Cndvar{n}{d} | |\Cndvar{n}{d-1}|$, and we may proceed by induction on $n$ and $d$ to show that for every $d\ge 2$, 
\begin{equation}
\label{eqn:cnd}
|\Cnd|\ge \left(\frac n{e^{\frac{1}{2}+\frac{1}{3}+\cdots+\frac{1}{d}}}\right)^{\binom{n-2}d} \ge \left( \frac{e^{1-\gamma}}{d+1}\cdot n\right) ^ {\binom {n-2}d},
\end{equation}
where $\gamma\approx 0.577$ is the Euler-Mascheroni constant.
For the induction step one needs to show that 
\[
\left(\frac {n}{e^{\frac{1}{2}+\frac{1}{3}+\cdots+\frac{1}{d}}}\right)^{\binom{n-2}d} 
  \left(\frac {n}{e^{\frac{1}{2}+\frac{1}{3}+\cdots+\frac{1}{d-1}}}\right)^{\binom{n-2}{d-1}}
\ge \left(\frac {n+1}{e^{\frac{1}{2}+\frac{1}{3}+\cdots+\frac{1}{d}}}\right)^{\binom{n-1}d}.
\]
To prove this step and to derive the right inequality in Equation (\ref{eqn:cnd}) use the fact that $e>(1+1/t)^{t}$ for every positive integer $t$.

Our main theorem improves the lower bound on $|\Tnd|$.
\begin{theorem}
\label{thm:1}
Let $t=t_d^*$ be the unique root in $(0,1)$ of
\[
(d+1)(1-t)+(1+dt)\ln t=0,
\]
and
\[
0>\alpha_d=\frac{1}{d+1} \int_{0}^{t_d^*}\frac{ (1-(1-y)^{d+1})\cdot \log{(1-(1-y)^{d+1})}\cdot(1-y+d\cdot y\log{y} ) }{y(1-y)^{d+1}}  dy.
\]
Then,
$$|\Tnd|\ge  \left( (1-o_n(1)) \frac{e^{1+\alpha_d}}{d+1}\cdot n\right) ^ {\binom {n-1}d}.$$ 


\end{theorem}

\begin{remark}
\begin{enumerate}
\item Theorem \ref{thm:1} offers an exponential improvement over Equation (\ref{eqn:cnd}) in every dimension. For example, for $d=2$ the lower bound is improved from approximately $(0.606\cdot n)^{n^2/2}$ to $(0.751\cdot n)^{n^2/2}$.
\item In addition, note that $\alpha_d\to 0$ as $d$ grows. Therefore, in contrast to the bound of Equation (\ref{eqn:cnd}), this lower bound on $|\Tnd|$ approaches Kalai's trivial {\em upper bound} of $\left(en/(d+1)\right)^{\binom{n-1}{d}}$ as $d$ grows.
\end{enumerate}
\end{remark}

It is interesting to speculate on whether Theorem \ref{thm:1} can help us prove Conjecture \ref{conj:main}, at least for large $d$. Namely, could it be that $|\Cnd|$ is even smaller than the lower bound of $|\Tnd|$ in the theorem? In particular, it is conceivable that $|\Cnd|$ does not approach the trivial upper bound as $d$ grows. This discussion naturally suggests the following quantitative version of Conjecture \ref{conj:main}
\begin{question}
Does it hold that $|\Cnd|/|\Tnd| = e^{-\Omega(n^d)}$ ? 
\end{question}
Put together, these two questions ask whether $|\Cnd|/|\Tnd|$ tends to zero and if so, whether the convergence is as fast as $e^{-\Omega(n^d)}$.

As we observe next, the upper bound of $|\Tnd|$ can be slightly improved.
\begin{theorem}\label{thm:upper}
For every dimension $d\ge 2$ there exists $\varepsilon_d>0$ such that
$|\Tnd| < \left( \frac {e-\varepsilon_d}{d+1}n\right) ^ {\binom {n-1}d}.$
\end{theorem}

As is often the case with the study of large combinatorial objects, we have a rather limited supply of interesting $d$-acyclic (i.e., having a trivial $d$-th homology) complexes, and $d$-hypertrees in particular. It is typically hard to analyze the boundary matrices of complexes that arise from combinatorial and probabilistic constructions. Obvious exceptions are $d$-collapsible complexes which are $d$-acyclic due to a purely combinatorial reason. Notable non-collapsible examples are the sum complexes, introduced in \cite{sum_complexes}, whose boundary operator has a useful analytical structure. In~\cite{LP} we used the theory of local weak convergence   to bound the dimension of the $d$-homology of random Linial-Meshulam complexes. This approach can work only when the (bipartite) incidence graph of $(d-1)$-faces vs.\ $d$-faces is locally a tree, in the sense of local weak convergence. Here we use similar techniques to construct large random $d$-complexes with a tiny $d$-homology.

We define next a random model $S_d(n,1)$ of $n$-vertex $d$-complexes with full $(d-1)$-skeleton. In this model each $(d-1)$-dimensional face $\tau$ independently chooses a vertex $v \notin\tau$ to form the $d$-face $v\tau$. Multi-faces are not allowed, and every $d$-face that is chosen more than once is counted only once. This model extends the well-studied random $1$-out graph, that is used in Wilson's ``cycle-popping'' algorithm to uniformly sample spanning trees \cite{Wilson}. Random $k$-out graphs were also studied in several additional combinatorial contexts \cite{shamir,frieze,bohman}

As usual, we say that a property holds asymptotically almost-surely (a.a.s.) if its probability tends to $1$ as $n\to\infty$. 

The next theorem shows that the random $1$-out $d$-complex $S_d(n,1)$ typically has a small top homology.
\begin{theorem}\label{thm:1out}
For every $d\ge 2$, a.a.s.\ the homology $H_d(S_d(n,1);\Q)$ has dimension $o(n^d).$
\end{theorem}

We also show that almost all the $d$-cycles in $S_d(n,1)$ can be eliminated  by removing each $d$-face independently with probability $\varepsilon$, for an arbitrarily small $\varepsilon>0$. We denote this random complex by $S_d(n,1-\varepsilon)$. Note that such complexes can be sampled as follows. Initially, let each $(d-1)$-face be {\em active} independently with probability $1-\varepsilon$. Then, each active $(d-1)$-face selects a random vertex to form a $d$-face as in $S_d(n,1)$.

It turns out that a random $d$-complex in the $S_d(n,1-\varepsilon)$ model is almost $d$-acyclic, in the sense that the only $d$-cycles it has are $\partial\Delta_{d+1}$, i.e., a boundary of $(d+1)$-simplex. Such $d$-cycles appear with positive probability that is bounded away from $1$.

\begin{theorem}\label{thm:1eps}
Fix an integer $d\ge 2$ and $\varepsilon>0$, and let $S$ be a random complex from $S_d(n,1-\varepsilon)$.
Then, a.a.s.\ $H_d(S;\Q)$ is generated by $\partial\Delta_{d+1}$'s, the number of which is Poisson-distributed with a bounded parameter. 
\end{theorem}

The rest of the paper is organized as follows. Section \ref{sec:back} contains some general background in simplicial combinatorics and basic facts on $d$-hypertrees. In Section \ref{sec:bounds} we prove the theorems regarding the enumeration of $d$-hypertrees, and Section \ref{sec:1out} is dedicated to the homology of the $1$-out random $d$-complex. Finally, we present various open questions in Section \ref{sec:open}.

\section{Background} 
\label{sec:back}
\subsection{Simplicial combinatorics}

A {\em simplicial complex} is comprised of a {\em vertex set} $V$ and a collection $X$ of subsets of $V$ that is closed under taking subsets. Namely, if $\sigma\in X$ and $\tau\subseteq \sigma$, then $\tau\in X$ as well. We usually refer to $X$ as the simplicial complex and call its members {\em faces} or {\em simplices}. The {\em dimension} of the simplex $\sigma\in X$ is defined as $|\sigma|-1$.
A $d$-dimensional simplex is also called a $d$-simplex or a $d$-face for short. The dimension $\dim(X)$ is defined as $\max\dim(\sigma)$ over all faces $\sigma\in X$. A $d$-dimensional simplicial complex is also referred to as a $d$-complex. The set of $j$-faces in $X$ is denoted by $F_j(X)$. For $t<\dim(X)$, the $t$-{\em skeleton} of $X$ is the simplicial complex that consists of all faces of dimension $\le t$ in $X$, and $X$ is said to have a {\em full} $t$-dimensional skeleton if its $t$-skeleton contains all the $t$-faces from $V$.  In this paper we usually work with a $d$-complex that has a full $(d-1)$-skeleton. 

For a face $\sigma$, the permutations on $\sigma$'s vertices are split in two
{\em orientations}, according to the permutation's sign.
The {\em boundary operator} $\partial=\partial_d$ maps an oriented
$d$-simplex $\sigma = (v_0,\ldots,v_d)$ to the formal sum
$\sum_{i=0}^{d}(-1)^i(\sigma^i)$, where $\sigma^i=(v_0,\ldots,v_{i-1},v_{i+1},\ldots,v_d)$ is an oriented
$(d-1)$-simplex. We fix some commutative ring $R$ and linearly extend
the boundary operator to free $R$-sums of
simplices. We denote by $\partial_d(X)$ the $d$-dimensional boundary operator of a $d$-complex $X$. Over the reals, the upper $(d-1)$-dimensional {\em Laplacian} of $X$ is $L=\partial_d(X)\partial_d(X)^*$.

When $X$ is finite, we consider the ${|F_{d-1}(X)|}\times{|F_d(X)|}$ matrix form of $\partial_d$ by choosing arbitrary orientations for
$(d-1)$-simplices and $d$-simplices. Note that changing the orientation of a $d$-simplex (resp. $d-1 $-simplex) results in multiplying the corresponding column (resp. row) by $-1$. 

Let $X$ be a $d$-dimensional simplicial complex. An element in the right kernel of the boundary operator $\partial_d(X)$ is called a {\em $d$-cycle}. Since $X$ is $d$-dimensional, the $d$-th homology group $H_d(X;R)$ (or vector space when $R$ is a field) of a $d$-complex $X$ equals to the space $\ker(\partial_d(X))$ of $d$-cycles. If $H_d(X;R)$ is trivial we say that $X$ is {\em $d$-acyclic}.

An element in the (right) image of $\partial_d$ is called a {\em $d$-boundary}. The $(d-1)$-st homology group $H_{d-1}(X;R)$ is the quotient group $\ker(\partial_{d-1}) / \im(\partial_d)$. Namely, the quotient of the $(d-1)$-cycles (=kernel of $\partial_{d-1}(X)$) and the $d$-boundaries.

Let us restrict the discussion to an $n$-vertex $d$-complex $X$ with full $(d-1)$-skeleton and the ring of rationals. Denote by $|X|=|F_d(X)|$ the number of top dimensional faces, $\beta_d(X)$ and $\beta_{d-1}(X)$ the dimensions of the corresponding homology groups (=Betti numbers), and $\rank(X)$ the rank of the operator $\partial_d=\partial_d(X)$.  Note that the Betti numbers and the rank are equal when we work over the rationals or the reals. By the rank-nullity theorem, $\beta_d(X) + \rank(X) = |X|$. In addition, $\beta_{d-1}(X)+\rank(X)=\binom{n-1}{d}$, since the space of $(d-1)$-cycles is $\binom {n-1}d$-dimensional (it is spanned by all the $(d-1)$-boundaries containing a specific vertex). These observations imply the following fact.
\begin{fact} 
Consider the three properties:
\begin{enumerate}
\item $|X|=\binom {n-1}d$.
\item $\beta_d(X)=0$.
\item $\beta_{d-1}(X)=0$.
\end{enumerate}
Then, any two properties imply the third. In such case, $X$ is a $d$-hypertree.
\end{fact}

Finally, the {\em homological shadow} $\SH(X)$, of a $d$-complex $X$, is the set of $d$-simplices $\sigma$ such that $\partial\sigma$ is a $d$-boundary of $X$. In other words, the complement $\overline\SH(X)$ is the set of $d$-simplices $\sigma$ for which $rank(X) < rank(X\cup\{\sigma\})$. The set $\overline\SH(X)$ plays a natural role in the construction of $d$-hypertrees, since it is comprised of those $d$-faces that can be added to the $d$-acyclic complex $X$ without creating a $d$-cycle.

\subsection{Linial-Meshulam complexes}
\begin{figure}[!b]
    \centering
    \subfloat[~The density of $\overline\SH(Y_2\left(n,\frac cn\right))$.]{{\includegraphics[width=0.45\textwidth]{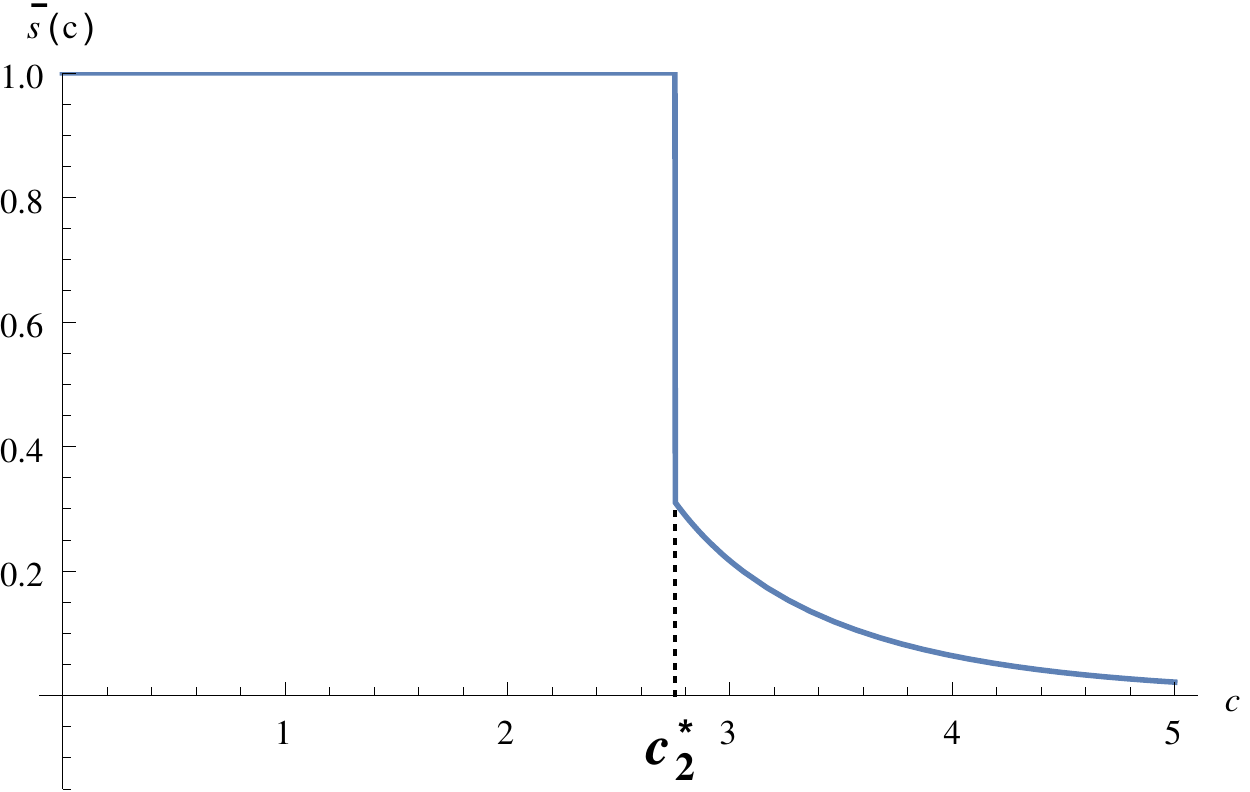} }\label{fig:shadowCompelx}}%
    \qquad
    \subfloat[~The normalized rank of $Y_2\left(n,\frac cn\right)$.]{{\includegraphics[width=0.45\textwidth]{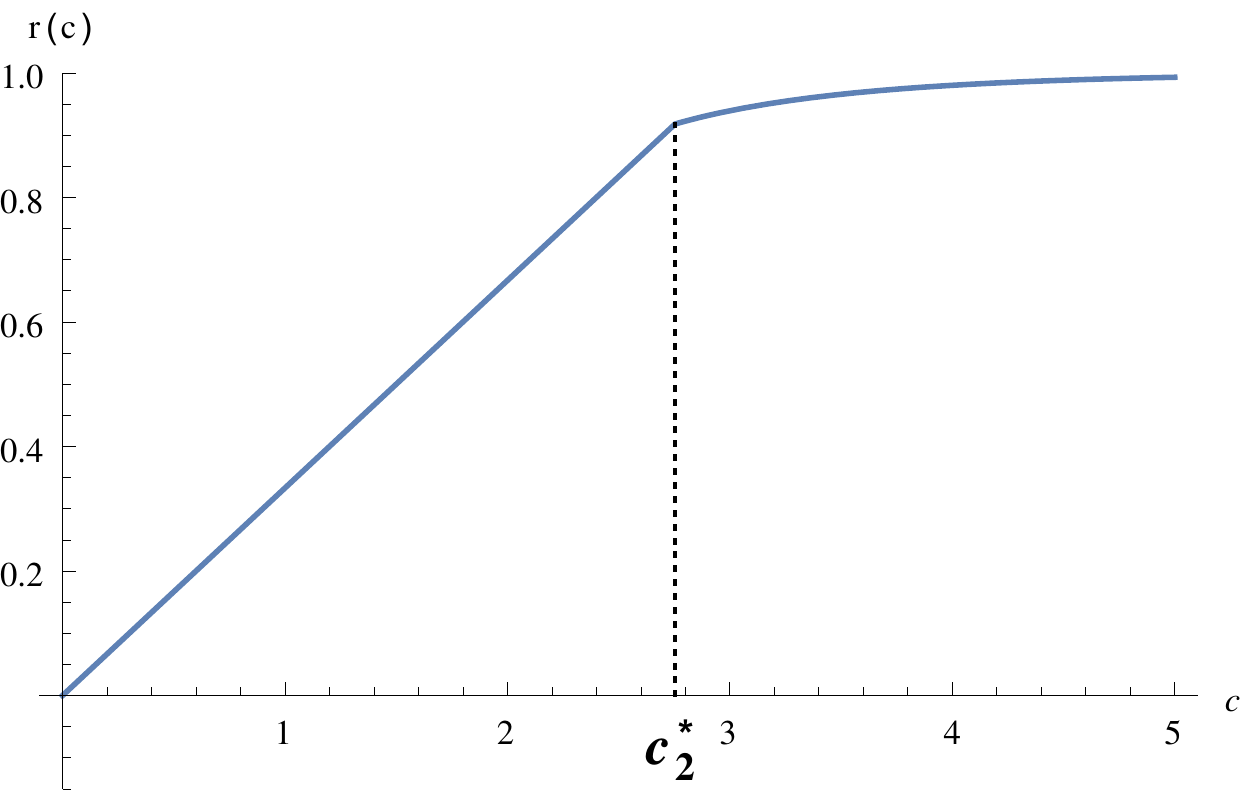} }\label{fig:shadowGraph}}%
    \caption{Plot of the functions $\bar s$ and $r$ for dimension $d=2$. }%
    \label{fig:shadowComparison}%
\end{figure}

The Linial-Meshulam complex $Y_d(n,p)$ is a random $n$-vertex $d$-dimensional simplicial complex with a full $(d-1)$-skeleton where every $d$-face appears independently with probability $p$. The topological invariants and combinatorial properties of these complexes have been intensively studied in recent years. This includes their homology groups, homotopy groups, collapsibility, embeddability and spectral properties. Here we use our previous paper~\cite{LP} that concerns the phase transition of this random simplicial complex, the threshold probability for $d$-acyclicity, and the emergence of a giant shadow. We need to briefly recall the pertinent results. Let $t=t_d^*$ be the unique root in $(0,1)$ of
\[
(d+1)(1-t)+(1+d\cdot t)\ln t=0,
\]
and let
\[c_d^* := \frac{-\ln t_d^*}{(1-t_d^*)^d}.\]

For $c>c_d^*$, let $t_c$ be the smallest positive root of $t=e^{-c(1-t)^d}$. Consider the functions $s,r:\R_{\ge 0}\to [0,1]$,
\[
\bar s(c) = 
\left\{ 
\begin{matrix}
1&c\le c_d^* \\
1-(1-t_c)^{d+1} & c>c_d^*
\end{matrix}
\right.~~~,~~~
r(c) = 
\left\{ 
\begin{matrix}
\frac{c}{d+1} &c\le c_d^* \\
\frac{c}{d+1}(1-(1-t_c)^{d+1})+(1-t_c)-ct_c(1-t_c)^d & c>c_d^*
\end{matrix}
\right.
\]
The function $r$ is strictly monotone, and we denote its inverse by $r^{-1}$. 
\begin{theorem}
Let $Y=\LMc$. Then, 
$$\mbox{(I)~} \lim_{n\to\infty}\frac{1}{\binom{n-1}d}\E[\rank(Y)]=r(c).$$
$$\mbox{(II) If $c\ne c_d^*$,~} \lim_{n\to\infty}\frac{1}{\binom{n}{d+1}}\E[|\overline\SH(Y)|]=\bar s(c).$$
In addition, for every $\varepsilon>0$, the probability that either the normalized rank or the density of the shadow's complement deviate from their expectation by more than $\varepsilon$ tends to $0$ as $n\to\infty$.
\end{theorem}
Note that the functions $\bar s$ and $r$ appear, with small variations in \cite{LP}. Namely, $1-\bar s(c)$ is the density of the shadow of $\LMc$, and $\frac{c}{d+1}-r(c)$ is its normalized $d$-dimensional Betti number, since the number of $d$-faces is $\approx c/(d+1)\binom{n-1}{d}$.

\subsection{Local weak convergence and $d$-trees}
We turn to describe the notion of local weak convergence of $d$-complexes. This concept is best described in the framework of rooted graphs. A rooted graph $(G,r)$ is comprised of a graph $G$ and a root vertex $r$. Two rooted graphs are considered isomorphic if there is a root-preserving isomorphism between them. 

Associated with a $d$-dimensional complex $X$, is the bipartite inclusion graph $G(X)$ between $X$'s set of $(d-1)$-faces $F_{d-1}(X)$ and its $d$-faces $F_d(X)$. A rooted $d$-complex $(X,o)$  is comprised of a $d$-complex $X$ and a $(d-1)$-face $o$ that is marked as its root. The graph $G(X)$ of a rooted $d$-complex $(X,o)$ is a rooted graph. For every integer $k\ge 0$, we denote by $(G(X),o)_k$ the rooted subgraph of $G(X)$ that is induced by vertices of {\em distance} at most $k$ from $o$ in $G(X)$.

Let $X_n$ be a sequence of random $d$-complexes and $(X,o)$ be a random rooted $d$-complex. Formally speaking, $X_n$ is a sequence of distributions on $d$-complexes and $(X,o)$ is a distribution on rooted $d$-complexes. We say that $(X,o)$ is the {\em local weak limit} of $X_n$ if for every integer $k>0$ and every rooted graph $(G,r)$,
\[
\Pr_{X_n,o_n}\left[ (G(X_n),o_n)_k\cong(G,r) \right]
\xrightarrow{n\to\infty}
\Pr_{(X,o)}\left[ (G(X),o)_k\cong(G,r) \right],
\]
where the root $o_n$ is sampled uniformly at random from $F_{d-1}(X_n)$.

We next define the concept of a {\em $d$-tree}. Do bear in mind that this is not to be confused with the notion of a $d$-hypertree. A $d$-tree $(B,o)$ is a rooted $d$-complex that can be viewed as a (possibly infinite) $d$-dimensional branching process. Initially the complex consists of the $(d-1)$-face $o$. At every step $k\ge 0$, every $(d-1)$-face $\tau$ of distance $2k$ from $o$ picks a non-negative number $m=m_\tau$ of new vertices $v_1,\ldots,v_m$, and adds the $d$-faces $v_1\tau,\ldots,v_m\tau$ to $B$. 

We observe that the graph $G(B)$ is a {\em rooted tree}, with $o$ as the root. Every vertex of odd depth (=distance from $o$) in $G(B)$ corresponds to a $d$-face and has exactly $d$ children. Every vertex $\tau$ of even depth has $m_\tau$ children. In fact, every rooted tree in which every vertex of odd depth has precisely $d$ children can be realized as an inclusion graph of a $d$-tree. Therefore, for every $(d-1)$-face $\tau\in B$ we refer to the {\em $d$-subtree rooted at $\tau$} as the rooted complex that contains $\tau$ and all its descendant faces. In particular, if the root $o$ is contained in $m$ $d$-faces $\sigma_1,\ldots,\sigma_m$, we denote by $B_{j,i}$ the $d$-subtree rooted at the $i$-th $(d-1)$-face of $\sigma_j$, for every $1\le j\le m$ and $1\le i \le d$.

The concept of local weak convergence is useful for us, since under some assumptions, the Betti numbers of a convergent sequence of finite complexes $X_n$ can be bounded by a parameter of its local weak limit. In addition, when the local weak limit is a $d$-tree, this parameter is expressible by some inductive formula. This approach is described in \cite{LP}, and we briefly mention the pertinent parts of that work.

Let $X$ be a (possibly infinite) $d$-complex. We are interested in the spectral measure of its upper $(d-1)$-dimensional Laplacian. This Laplacian $L$ is a symmetric operator acting on the Hilbert space $\mathcal H=\ell^2(F_{d-1}(B))$.
For finite $X$, we can apply the spectral theorem to $L$ which is symmetric, and therefore self-adjoint. However, for infinite $X$ the situation is more subtle. The Laplacian $L$ is only densely-defined on the subspace of functions with finite support. It has a unique extension $\hat L$ to $\mathcal H$ which may be (but is not necessarily) self-adjoint. We say that a complex is self-adjoint if the extension $\hat L$ is a self-adjoint operator. For example, if $B$ is a random rooted $d$-tree such that the expected degree of its $(d-1)$-faces is bounded, then it is almost surely self-adjoint. When $X$ is self-adjoint, we can apply the {\em spectral theorem} and obtain the {\em spectral measure} $\mu_{X,\tau}$ of $\hat L$ with respect to the characteristic vector of a $(d-1)$-face $\tau$. 

A key parameter that we study is $x_B:=\mu_{B,o}(\{0\})$. In words, $\mu_{B,o}$ is the spectral measure of the (extended) Laplacian of a $d$-tree $B$ with respect to its root $o$. The reason that we are interested in the measure of the atom $\{0\}$ is that for a finite complex $X$, the sum $\sum_{\tau\in F_{d-1}(X)}\mu_{X,\tau}(\{0\})$ is equal to the dimension of the kernel of $X$'s Laplacian, which is very close to $\beta_{d-1}(X)$. Here are two key lemmas from \cite{LP} that we use. 


\begin{lemma}\label{lem:aux_lim}
Let $X_n$ be a random $n$-vertex $d$-complex with a full $(d-1)$-skeleton, and $(B,o)$ be a random rooted $d$-tree that is almost surely self-adjoint. If $X_n$ locally weakly converges to $(B,o)$ then,
\[
\limsup_{n\to\infty}\frac{1}{\binom nd}\E_{X_n}[\beta_{d-1}(X_n)]\le \E_{B}[x_B]
\]
\end{lemma}
We can, in fact, say more. Namely, the expected spectral measure of $X_n$'s Laplacian with respect to a uniformly random root weakly converges to the expected spectral measure of $B$'s Laplacian. The following lemma enables us to compute the seemingly complicated parameter $x_B$.
\begin{lemma}\label{lem:aux_ind}
Let $B$ be a self adjoint $d$-tree. Suppose that the root $o$ is contained in $m$ $d$-faces and consider the rooted $d$-subtrees $B_{j,i},~~1\le j\le m$,~ $1\le i \le d$, as defined above. If there exists a $1\le j \le m$ such that 
$
x_{B_{j,1}}=\cdots=x_{B_{j,d}}=0, 
$
then $x_B=0$. Otherwise,
\[
\frac{1}{x_B} = 1 + \sum_{j=1}^{m}\frac 1{\sum_{i=1}^d x_{B_{j,i}}}.
\]
\end{lemma}

\section{Enumeration of $d$-hypertrees}
\label{sec:bounds}
\subsection{Proof of Theorem \ref{thm:1}}
We start with the following extremal question: What is the largest possible shadow of an $n$-vertex $d$-acyclic complex $X$ with a given number of $d$-faces? Equivalently, what is the least possible number of ways that $X$ can be extended to a $d$-acyclic complex with $(|X|+1)$ $d$-faces? Although the following claim is not tight, it suffices for our purposes and its proof is fairly simple. It is possible to derive a tight bound using shifting methods~\cite{BK}.
\begin{claim}
\label{clm:worst_case}
If $X$ is an $n$-vertex $d$-acyclic complex, then $|\overline{\SH}(X)|\ge \binom n{d+1}\left(1 - \frac{|X|}{\binom{n-1}d} \right)$.
\end{claim}
\begin{proof}
Denote $Y=\SH(X)$ and observe that $\rank(Y)=|X|$ since $X$ is $d$-acyclic. There are clearly exactly $(d+1)|Y|$ pairs $(v,\sigma)$ where $v\in\sigma$ is a vertex and $\sigma\in Y$ a $d$-face. In addition, consider the collection of all $d$-faces of $Y$ that contain $v$. This is an acyclic subcomplex, since even the $d$-complex with {\em all} $\binom{n-1}{d}$ $d$-faces that contain $v$ is acyclic. Therefore, the number of such $d$-faces is at most $\rank(Y)=|X|$. It follows that $(d+1)|Y| \le n\cdot|X|$, as claimed.
\end{proof}

As we observe next, Claim \ref{clm:worst_case} yields a simple proof of Kalai's lower bound.
\begin{corollary} [\cite{kalai}]
\label{cor:kalai_simple}
$|\Tnd|\ge \left( \frac {n}{d+1}\right) ^ {\binom {n-1}d}.$
\end{corollary}
\begin{proof}
Let us construct a $d$-hypertree starting with a full $(d-1)$-skeleton and adding one $d$-face at a time. By Claim \ref{clm:worst_case} there are at least $\binom n{d+1}\left(1 - \frac{i}{\binom{n-1}d} \right)$ possible choices for our $i$-th step. This argument counts every $d$-hypertree $\binom{n-1}d !$ times. Therefore,
\[
\binom{n-1}d !~\cdot~|\Tnd|  \ge \prod_{i=0}^{\binom{n-1}d -1}\binom n{d+1}\left(1 - \frac{i}{\binom{n-1}d} \right).
\]
The claim follows directly.
\end{proof}

We turn to prove Theorem \ref{thm:1} in a way that refines the previous argument. Consider the {\em random $d$-acyclic complex process} on $n$ vertices $\mathbf T_d(n)=T_0\subset T_1 \subset \ldots\subset T_{\binom{n-1}d}$, where $T_0$ is the full $(d-1)$-skeleton on $n$ vertices. In every step $1\le i \le \binom{n-1}d$, a $d$-face $\sigma_i$ is sampled uniformly at random from $\overline{\SH}(T_{i-1})$  and added to the complex, i.e., $T_i=T_{i-1}\cup\{\sigma_i\}$. Since this process produces a random ordered $d$-hypertree, its support size is $\binom{n-1}d!\cdot |\mathcal T_{n,d}|$, and its entropy does not exceed the logarithm of this number. But this entropy can actually be computed:
\begin{equation}
\label{eqn:entropy}
H(\mathbf T_d(n)) = \sum_{i=0}^{\binom{n-1}d-1} H(\sigma_{i+1} ~|~ T_{i}) = \sum_{i=0}^{\binom{n-1}d-1} \E_{T_i}[\log \left(|\overline{\SH}(T_{i})| \right)].
\end{equation}
The first equality is the chain rule for entropy, and the second equality follows since $\sigma_{i+1}$ is selected uniformly from $\overline{\SH}(T_{i})$.

The random process $\mathbf T_d(n)$ is closely related to the Linial-Meshulam model. Indeed, consider a random ordering $\vec\sigma=\sigma_1,\ldots,\sigma_{\binom n{d+1}}$ of all $d$-faces over $n$ vertices. The $d$-faces $\sigma_1,\ldots,\sigma_M$ along with a full $(d-1)$-skeleton on $n$ vertices constitute the complex $Y_d(n,M)$. Therefore, $\vec\sigma$ can be also used to define the complex $Y_d(n,p)$ which equals to $Y_d(n,\hat M)$, where $\hat M$ $\mbox{Bin}(\binom n{d+1},p)$-distributed. Let us call an index $1\le j \le \binom n{d+1}$ {\em critical} if $\sigma_j\in\overline\SH(Y_d(n,j-1))$. The complex $\mathbf T_d(n)$ equals to $\{ \sigma_j ~|~\mbox{$j$ is critical}\}$ with a full $(d-1)$-skeleton, and moreover, $T_i$ is its subcomplex containing the first $i$ critical faces $\sigma_{j_1},\ldots,\sigma_{j_i}$ in $\vec\sigma$.

Recall that the function $r=r(c)$ depicts the normalized rank of $\LMc$ as defined in Section \ref{sec:back}. If $i=i(n)$ is an integer and $c>0$ real such that $r(c) - i/\binom{n-1}{d}$ is positive and bounded away from zero, we may assume that a.a.s. $T_i \subset \LMc$. Indeed, in such case $\LMc$ a.a.s. has rank greater than $i$, and therefore contains more than $i$ critical faces.
In particular, a.a.s.\ ,
\begin{equation}
\label{eqn:sh_coupling}
|\overline\SH(T_i)| \ge \left|\overline\SH\left(\LMc\right)\right|.
\end{equation}

As explained in Section \ref{sec:back}, the size $\left|\overline\SH\left(\LMc\right)\right|$ is concentrated at $\bar s(c)\binom n{d+1}$. In the proof below we use this accurate estimation {\em for every $c>0$} to deduce the bound stated at Theorem \ref{thm:1}. However, note that even a substantially simpler argument already yields a pretty good bound. Namely, since $\bar s(c)=1$ for $c<c_d^*$, it follows that each of the first $\frac{c_d^*-\varepsilon}{d+1}\binom{n}{d}$ summands in Equation (\ref{eqn:entropy}) equals $(1-o(1))\log{\binom{n}{d+1}}$, for $\varepsilon>0$ arbitrarily small. The other summands in that equation can be easily bounded by the worst-case analysis of Claim \ref{clm:worst_case}. It turns out that for every dimension $d\ge 2$, this simple argument improves the bound of Equation (\ref{eqn:cnd}). In addition, since $c_d^*=d+1-o_d(1),$ it gives a bound of the form $$|\Tnd|\ge  \left( \frac{e^{1-\varepsilon_d}}{d+1}\cdot n\right) ^ {\binom {n-1}d},$$ where $0<\varepsilon_d\to 0$ as $d$ grows. We do not go into further details of this argument, since we derive below a better lower bound by a more careful analysis.

Denote $\lambda_i:=\E_{T_i}[\log \left(|\overline{\SH}(T_{i})| \right)]-\log\binom{n}{d+1}$. We split the summation in Equation (\ref{eqn:entropy}) to four parts.

\begin{claim}
Let $\varepsilon>0$. 
\begin{enumerate}
\item {\em Subcritical:} $~0\le i \le \left(\frac{c_d^*}{d+1}-\varepsilon\right)\binom{n-1}{d}$.
\[
\lambda_i \ge (1-\varepsilon)\log(1-\varepsilon) + \varepsilon\log\left(1-\frac{c_d^*}{d+1}\right).
\]
\item {\em Transition:} $~\left(\frac{c_d^*}{d+1}-\varepsilon\right)\binom{n-1}{d} < i \le \left(\frac{c_d^*}{d+1}+\varepsilon\right)\binom{n-1}{d}$ then,
\[
\lambda_i \ge \log\left(1-\frac{c_d^*}{d+1}-\varepsilon\right).
\]
\item {\em Superctitical:} $~\left(\frac{c_d^*}{d+1}+\varepsilon\right)\binom{n-1}{d} < i \le (1-2\sqrt\varepsilon)\binom{n-1}{d}$ then,
\[
\lambda_i \ge (1-\varepsilon)\log\left(\bar s\left(r^{-1}\left(\frac i{\binom{n-1}{d}}\right)\right)\right) - (1-\varepsilon)\cdot2\sqrt\varepsilon + \varepsilon\log(2\sqrt\varepsilon).
\]
\item {\em Rearguard:} $~(1-2\sqrt\varepsilon)\binom{n-1}{d} < i <\binom{n-1}d$ then,
\[
\lambda_i \ge \log\left( 1-\frac{i}{\binom{n-1}d} \right)
\]
\end{enumerate}
\end{claim}

\begin{proof}
For the Transition and Rearguard ranges, the inequality is just that of Claim \ref{clm:worst_case}. For the Subcritical range, we fix some $c$ with $\frac{c_d^*}{d+1}-\varepsilon<c<\frac{c_d^*}{d+1}$ and  apply Equation (\ref{eqn:sh_coupling}). This implies that with probability at least $1-\varepsilon$, there holds
$|\overline\SH(T_i)| > \binom{n}{d+1}(1-\varepsilon)$, since $\LMc$ has a.a.s.\ a shadow of vanishingly small density. By Claim \ref{clm:worst_case}, $|\overline\SH(T_i)| > \binom{n}{d+1}(1-c_d^*/(d+1))$ always holds. Therefore,
\[
\lambda_i \ge (1-\varepsilon)\log\left(1-\varepsilon \right) + \varepsilon\log\left(1-c_d^*/(d+1) \right).
\]
We turn to consider the Supercritical range. Let $~\left(\frac{c_d^*}{d+1}+\varepsilon\right)\binom{n-1}{d} < i \le (1-2\sqrt\varepsilon)\binom{n-1}{d}$ and $c_i=r^{-1}\left(\frac i{\binom{n-1}{d}}\right)>c_d^*$. By using Claim \ref{clm:worst_case} as in the Subcritical item, it suffices to show that there exists $c>0$ such that (i) $r(c) - i/\binom{n-1}{d}\gg 0$ and (ii) a.a.s.,
\[
\log\left|\overline\SH\left(\LMc\right)\right| \ge \log\binom n{d+1} +\log(\bar s(c_i))- 2\sqrt\varepsilon.
\]
Since $\bar s$ is continuous when $c>c_d^*$ and since $r$ is strictly monotone, we may choose $c>c_i$ such that both $\bar s(c)>\bar s(c_i)-\varepsilon$ and condition (i) is satisfied. In addition, a.a.s.\ $\left|\overline\SH\left(\LMc\right)\right| \ge \binom{n}{d+1}(\bar s(c) - \varepsilon)$, and therefore,
\[
\log\left|\overline\SH\left(\LMc\right)\right| - \log\binom{n}{d+1} >  \log(\bar s(c_i) - 2\varepsilon) \ge \log(\bar s(c_i)) - 4\varepsilon/\bar s(c_i),
\]
where the last inequality is by straightforward analysis. By Claim \ref{clm:worst_case}, $\bar s(c_i) \ge 1-i/\binom{n-1}d \ge 2\sqrt\varepsilon$. which concludes the proof.
\end{proof}

Let us estimate the sum $\sum_{i=0}^{\binom {n-1}d -1} \lambda_i$. We use the notation $o_\varepsilon(1)$ for terms that vanish as $\varepsilon\to 0$. Since every index in the Subcritical regime contributes $o_\varepsilon(1)$ to the sum, the entire range contributes only $\binom {n-1}d \cdot o_\varepsilon(1)$. The same upper bound applies as well to the Transition regime, which contains $2\varepsilon\binom {n-1}d$ bounded terms. The entire contribution of the Rearguard regime equals to 
\[
\log \left( \frac{\left( 2\sqrt\varepsilon\binom{n-1}d \right) !}{\binom{n-1}d^{\binom{n-1}d}} \right) = \binom{n-1}d\cdot o_\varepsilon(1).
\]
In addition, the last two terms in each $\lambda_i$ in the Supercritical regime are of order $o_\varepsilon(1)$ as well. We conclude that
\[
H(\mathbf T_d(n)) \ge \binom{n-1}{d}\left(\log\binom n{d+1} + (1-\varepsilon)\frac{1}{\binom{n-1}{d}}\sum_{i=\frac{c_d^*}{d+1}+\varepsilon }^{1-2\sqrt\varepsilon}\log\left(\bar s\left(r^{-1}\left(\frac i{\binom{n-1}{d}}\right)\right)\right) +o_\varepsilon(1) \right).
\]
On the other hand, the entropy cannot exceed the logarithm of the cardinality of the support, i.e.,
\[
H(\mathbf T_d(n))\le \log |\mathcal T_{d,n}| + \binom{n-1}{d}\left(\log \binom{n-1}{d} - 1\right).
\]
Letting $\varepsilon\to 0$ and $\alpha_d:=\int_{c_d^*/{d+1}}^{1} \log\bar s(r^{-1}(x))dx$ yields

$$|\Tnd|\ge  \left( (1-o_n(1)) \frac{e^{1+\alpha_d}}{d+1}\cdot n\right) ^ {\binom {n-1}d}.$$ 
In order to complete the proof, we need to establish the integral form for $\alpha_d$ as stated in the theorem. Consider the function $t$ that maps every $c>c_d^*$ to the smallest positive root of $t=e^{-c(1-t)^d}.$  The derivative of $t$ w.r.t.\ $c$ is $t' = -\frac{t(1-t)^{d+1}}{1-t+dt\ln{t}}$ (See \cite{LP}). In addition, recall that for $c>c_d^*,~~\bar s(c)= 1-(1-t(c))^{d+1}$ and a straightforward computation yields that $r'(c)=\bar s(c)/(d+1).$ The desired integral form is obtained by the change of variables $y=t(r^{-1}(x))$.

\subsection{Proof of Theorem \ref{thm:upper}}
Kalai's upper bound $|\Tnd| < \left( \frac {e}{d+1}\cdot n\right) ^ {\binom {n-1}d}$ accounts for {\em all} $n$-vertex $d$-complexes with $\binom {n-1}d$ $d$-faces.
We slightly improve this bound by estimating the probability that a uniformly sampled complex with these parameters is a $d$-hypertree. Theorem \ref{thm:upper} immediately follows from the following Lemma.
\begin{lemma}
The probability that $Y:=Y_d\left(n,\binom{n-1}d \right)$ is a $d$-hypertree is $\exp(-\Omega(n^d))$.
\end{lemma}
\begin{proof}
Let $\sigma_1,\sigma_2,\ldots$ be an infinite sequence of $d$-faces on a set of $n$ vertices, each chosen {\em independently} uniformly at random. 
Consider the $n$-vertex complex $\tilde Y_d\left(n,M \right)$ that has a full $(d-1)$-skeleton and the first $M$ $d$-faces in the sequence, where repetitions of the same face get removed. The complex $Y$ is equivalent to $\tilde Y_d(n,\bar M)$ where $\bar M$ is the (random) index for which the prefix $\sigma_1,\ldots,\sigma_{\bar M}$ contains exactly $\binom{n-1}d$ distinct $d$-faces. 

Let $\varepsilon>0$ be a small constant. If $M=(1+\frac{\varepsilon}{d+1})\binom{n-1}d$ then $Y$ is contained in $\tilde Y_d(n,M)$ with probability $1-\exp(-\Theta(n^d))$ by a standard measure concentration argument. Let us denote by $R$ the rank of $\tilde Y_d(n,M)$.

The probability that $Y$ is a $d$-hypertree is bounded by $\exp(-\Theta(n^d))$ plus the probability that $\tilde Y_d(n,M)$ has full rank, i.e., $R=\binom{n-1}d$. In addition, $\tilde Y_d(n,M)$ is a.a.s.\ contained in $Y_d(n,p)$ for $p=\frac{d+1+2\varepsilon}n$, since a $\mbox{Bin}\left(\binom n{d+1},p\right)$ random variable is a.a.s.\ greater than $M$. Therefore, the expectation $\E[R]$ is bounded, up to an additive error term of $o(n^d)$, by the expected rank of $Y_d(n,p)$ which equals to $\binom{n-1}d\cdot r(d+1+2\varepsilon)$. Since $r(d+1)<1$, we can choose $\varepsilon$ so that $\E[R]<(1-\delta)\binom{n-1}d$ for some constant $\delta>0$.
The proof is concluded by observing that $R$ is a $1$-Lipschitz function that depends on $M=(1+\frac{\varepsilon}{d+1})\binom{n-1}d$ independent variables. By Azuma's inequality \cite{Azuma},
\[
\Pr\left[R=\binom{n-1}d\right]\le \exp\left(- \frac{\delta^2}{2\left(1+\frac{\varepsilon}{d+1}\right)}\binom{n-1}d \right).
\]
\end{proof}

\section{The random $1$-out $d$-complex}
\label{sec:1out}
The random $1$-out $d$-complex $S_d(n,1)$  is an $n$-vertex $d$-dimensional complex with a full $(d-1)$-skeleton, in which every $(d-1)$-face $\tau$ selects, independently uniformly at random, a $d$-face $\sigma_\tau$ that contains it. To wit, $\tau$ selects a uniform random vertex $v_\tau\notin\tau$ and $\sigma_\tau=v_\tau\tau$. The selection process is done independently, but we remove multiply selected $d$-faces to maintain a simplicial complex. 
The purpose of this section is to show that the $d$-dimensional homology of $S_d(n,1)$ is a.a.s.\ of dimension $o(n^d)$. Moreover, the complex can be made very close to $d$-acyclic by a random sparsification. The upper bound on the dimension of the homology is proved using the spectral measure of the local weak limit as presented in Section \ref{sec:back}.

\subsection{The local weak limit of the random $1$-out process}
We first describe the local weak limit of $S=S_d(n,1)$. That is, the limiting distribution of local neighborhoods of a root $(d-1)$-face in $S$. This local weak limit is a random $d$-tree that we denote by $\textbf{B}_d$. The number of children $m_\tau$ of a $(d-1)$-face $\tau$ in $\textbf{B}_d$ is independently distributed, but the $(d-1)$-faces and $d$-faces in $\textbf{B}_d$ come in two types - $(A)$ and $(B)$. The type of every $(d-1)$-face $\tau$ always coincides with that of its selected $d$-face $\sigma_\tau$. The type indicates whether, when exposing the neighborhood of the root, the selected $d$-face $\sigma_\tau$ appeared $(A)$ after or $(B)$ before the selecting $(d-1)$-face $\tau$. For instance, the root is always of type $(A)$.

For a $(d-1)$-face $\tau$ of type $(A)$, the number of children $m_\tau$ is $1+\mbox{Poi}(d)$ distributed. One descendant $d$-face is $\sigma_\tau$ whose type is also $(A)$, whence all its $(d-1)$-subfaces also have type $(A)$. The other $\mbox{Poi}(d)$ descendant $d$-faces have type $(B)$. These are the $d$-faces that contain $\tau$ but were selected by some other $(d-1)$-face. Therefore, a type $(B)$ $d$-face $\sigma=\sigma_\rho$ has precisely one descending $(d-1)$-face $\rho$ of type $(B)$ whereas all others have type $(A)$. On the other hand, for $\tau$ of type $(B)$, the distribution of $m_\tau$ is $\mbox{Poi}(d)$, with all descending $d$-faces of type $(B)$ (See Figure \ref{fig:dtree}).

To generate a $d$-tree $\textbf{B}_d$, we start with a root $(d-1)$-face of type $(A)$, generate its descendant $d$-faces, and keep track of the types of the new $(d-1)$-faces. At each step $k>0$, we generate the descendants of the $(d-1)$-faces at distance $2k$ from the root according to their type.

For $d=1$, $\textbf{B}_1$ is a well-known random tree model. Type $(A)$ vertices form an infinite rooted path, and every such vertex ``grows" a Galton-Watson Poi$(1)$ branching process with vertices of type $(B)$. This is known to be the local weak limit of a uniform spanning tree ~\cite{grimmett}.

\begin{figure}[!h]
\begin{center}
\begin{tikzpicture}
\begin{scope}[xshift=-4cm]
\node [circle,draw] (dm1A)					{A};
\node [rectangle,draw] (dA)	[below=of dm1A, xshift=-2.5cm]		{A} edge [<-,ultra thick] (dm1A);
\node [circle,draw] (dm1A1)	[below=of dA, xshift=-1cm]	{A} edge [-] (dA); 
\node [circle,draw] (dm1A2)	[below=of dA, , xshift=-.2cm]				{A} edge [-] (dA); 
\node [circle,draw] (dm1A1)	[below=of dA, xshift=0.6cm]	{A} edge [-] (dA); 

\node [rectangle,draw] (dB)	[below=of dm1A]		{B} edge [-,dashed] (dm1A);
\node [circle,draw] (dm1B1)	[below=of dB, xshift=-0.8cm]	{B} edge [->,ultra thick] (dB); 
\node [circle,draw] (dm1A3)	[below=of dB, , xshift=+0.0cm]				{A} edge [-] (dB); 
\node [circle,draw] (dm1A3)	[below=of dB, , xshift=+0.8cm]				{A} edge [-] (dB); 

\node [rectangle,draw] (dB1)	[below=of dm1A,xshift=2.5cm]		{B} edge [-,dashed] (dm1A);
\node [circle,draw] (dm1B2)	[below=of dB1, xshift=-0.5cm]	{B} edge [->,ultra thick] (dB1); 
\node [circle,draw] (dm1A4)	[below=of dB1, , xshift=+0.3cm]				{A} edge [-](dB1); 
\node [circle,draw] (dm1A4)	[below=of dB1, , xshift=+1.1cm]				{A} edge [-](dB1); 
 \node[] (txt) [right=of dB,xshift=-1.3cm,yshift=0.7cm]	{$\mbox{Poi}(d)$};
\end{scope}
   \begin{scope}[xshift=+4cm]
\node [circle,draw] (dm1B)					{B};
\node [rectangle,draw] (dB1)	[below=of dm1B,xshift=-1.1cm]		{B} edge [-,dashed] (dm1B);
\node [circle,draw] (dm1B2)	[below=of dB1, xshift=-1cm]	{B} edge [->,ultra thick] (dB1); 
\node [circle,draw] (dm1A4)	[below=of dB1, , xshift=-0.2cm]				{A} edge [-](dB1); 
\node [circle,draw] (dm1A5)	[below=of dB1,xshift=+0.6cm]				{A} edge [-](dB1); 

\node [rectangle,draw] (dB1)	[below=of dm1B,xshift=1.1cm]		{B} edge [-,dashed] (dm1B);
\node [circle,draw] (dm1B2)	[below=of dB1, xshift=-0.5cm]	{B} edge [->,ultra thick] (dB1); 
\node [circle,draw] (dm1A4)	[below=of dB1, , xshift=+0.3cm]				{A} edge [-](dB1); 
\node [circle,draw] (dm1A5)	[below=of dB1,xshift=+1.1cm]				{A} edge [-](dB1); 
 \node[] (txt) [below=of dm1B,yshift=0.7cm]	{$\mbox{Poi}(d)$};
\end{scope}
\end{tikzpicture}
\end{center}
\caption{Illustration of the inclusion graph of the $d$-tree $\textbf{B}_d$ for $d=3$. Circles represent $(d-1)$-faces, and squares are $d$-faces, each vertex is marked with the face's type. The thick arrows are included only to explain how the selection process of $S_d(n,1)$ is reflected in the $d$-tree. A dashed line demonstrate that the number of type $(B)$ $d$-faces descending from a $(d-1)$-face is $\mbox{Poi}(d)$ randomly distributed. Since the root has type $(A)$, the information in this figure completely determines the distribution of $\textbf{B}_d$.}
\label{fig:dtree}
\end{figure}
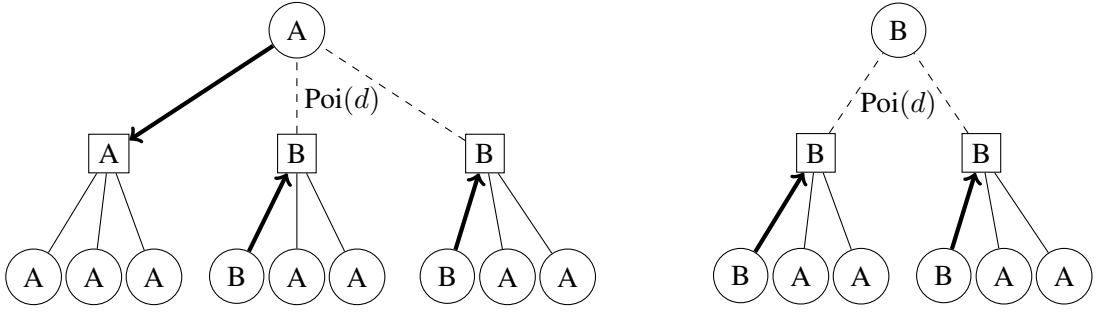

\begin{claim}
$\textbf{B}_d$ is the local weak limit of $S_d(n,1)$.
\end{claim}
\begin{proof}
It is easy to observe that a.a.s.\ all the $(d-1)$-faces in $S$ have degree $O(\log n)$. Indeed, up to negligible duplications, the degree of a $(d-1)$-face $\tau$ is one plus the number of $(d-1)$-faces that chose a $d$-face containing $\tau$, which is $\mbox{Bin}\left( (n-d)d,\frac{1}{n-d}\right)$-distributed. 

Fix some root $(d-1)$-face $o\in S$ and an integer $k>0$. Both the number of $(d-1)$-faces and vertices that appear in a bounded-radius neighborhood of $o$ are a.a.s.\ polylogarithmic in $n$. The local $d$-tree structure is disrupted only if in generating $S$, a $(d-1)$-face at a bounded distance from $o$ selects a vertex in its bounded radius neighborhood. This, however, is very unlikely to occur, so that only with probability $o_n(1)$ is the $k$-radius neighborhood not a $d$-tree. 

It remains to show that the probability of every fixed $d$-tree of depth at most $k$ tends to its probability in the distribution $\textbf{B}_d$. Since the $k$-local neighborhood of the root $o\in S$ is a.a.s.\ a $d$-tree, we can expose it as in a generative process of a $d$-tree. By the construction of $\textbf{B}_d$, there is only one difference between the exposure process of $o$'s neighborhood in $S$ and the generative process of $\textbf{B}_d$. Namely, the independent Poisson$(d)$ variables in $\textbf{B}_d$ are replaced by possibly dependent Binomial variables in $S$. 

Suppose we have already exposed some part of $o$'s  neighborhood in $S$, and we are about to expose the descendant $d$-faces of some $(d-1)$-face $\tau$. 
The situation varies according to whether $\tau$'s selected $d$-face $\sigma_\tau$ has already been exposed, but in this respect there is no difference between the two processes. The difference is that in $\textbf{B}_d$, $\tau$ has a $\mbox{Poi}(d)$ independently distributed number of descendant $d$-faces of type $(B)$. These correspond to the $d$-faces in $S$ that contain $\tau$, that were selected by some other $(d-1)$-face but have not yet been exposed. In $S$, this  number is distributed binomially, where the number of trials is a.a.s.\ $nd-o(n)$ and the success probability is $(1-o(1))/n$. In particular, as $n\to\infty$, this number tends to a $\mbox{Poi}(d)$ variable. Note that in both parameters of this binomial distribution, the error term may depend on the already exposed neighborhood of $o$, so the different degrees could be dependent. However, since this dependency a.a.s.\ only affects the error term, the joint distribution of all these numbers tends to the distribution of independent Poisson variables.
\end{proof}

\subsection{Proof of Theorem \ref{thm:1out}}
Let $S=S_d(n,1)$. Theorem \ref{thm:1out} would follow if we can show that $\E[\beta_{d-1}(S)]=o(n^d)$. Indeed, as we saw in Section \ref{sec:back}, 
$
\beta_d(S)=\beta_{d-1}(S)-\left(|S|-\binom{n-1}{d}\right),
$
and a.a.s.\ $|S|=\binom{n}{d}-o(n^d)$ since there are only $o(n^d)$ duplications. In addition, by Markov's inequality, if $\E[\beta_{d}(S)]=o(n^d)$ then a.a.s.\  $\beta_d(S)=o(n^d)$.

Recall that we denote by $x_{B,o}\in [0,1]$ the spectral measure of the Laplacian of a $d$-tree $B$ with respect to the characteristic vector of its root $o$, measured at the atom $\{0\}$. By Lemma \ref{lem:aux_lim}, 
$$\limsup_{n\to\infty}\frac{1}{\binom{n}{d}}\E[\beta_{d-1}(S)]\le\E[x_{\textbf{B}_d}],
$$
since $S$ locally weakly converges to $\textbf{B}_d$. Therefore, Theorem \ref{thm:1out} follows from the following lemma.
\begin{lemma}
$\E[x_{\textbf{B}_d}]=0.$
\end{lemma}
\begin{proof}
We denote the random variable $X=x_{\textbf{B}_d}$. In addition, let $\textbf{B}^{(B)}_d$ denote a random $d$-tree that is generated exactly like $\textbf{B}_d$ except that the root is of type $(B)$, and we denote the random variable $Y=x_{\textbf{B}^{(B)}_d}$. In addition, we denote $a:=\Pr[X>0]$ and $b:=\Pr[Y>0]$. We also use $X_1,X_2,\ldots$ or $Y_1,Y_2,\ldots$ to denote i.i.d copies of $X$ and $Y$ respectively.

By Lemma \ref{lem:aux_ind}, $X$ has the distribution obtained by the following process. We independently sample $X_1,\ldots,X_d$, a Poi$(d)$ distributed number $m$, and for every $1\le j\le m$, we also sample $Y_j$ and $X_{j,1},\ldots,X_{j,d-1}$. The r.v.\ $X$ equals $0$ if $X_1=\cdots=X_d=0$ or $Y_j=X_{j,1}=\cdots=X_{j,d-1}=0$ for some $j$. It is otherwise computed by the formula in Lemma \ref{lem:aux_ind}:
\[
X=\left(1+\frac 1{X_1+\cdots+X_d}+\sum_{j=1}^{m}\frac 1{Y_j+X_{j,1}+\cdots+X_{j,d-1}}\right)^{-1}.
\]
Clearly, a similar distributional equation can be derived for $Y$.

First, we use these distributional equations in order to derive relations between $a$ and $b$. The probability that $Y_j=X_{j,1}=\cdots=X_{j,d-1}=0$ for some given $j$ is $(1-b)(1-a)^{d-1}.$ Therefore, the probability that for every $1\le j\le m$ this does not hold, where $m\sim \mbox{Poi}(d)$, equals to $e^{-d(1-b)(1-a)^{d-1}}$. In addition, the probability that $X_1=\cdots=X_d=0$ equals to $(1-a)^{d-1}$. Therefore,
$$ a=(1-(1-a)^d)e^{-d(1-b)(1-a)^{d-1}}.$$
By a similar argument, $b=e^{-d(1-b)(1-a)^{d-1}}$, so that
\begin{equation}\label{eqn:ab}
a=(1-(1-a)^d)b.
\end{equation}

Let $P$ be a random variable whose distribution is that of $X_1+\cdots+X_d$, and let $Q,Q_1,Q_2,\ldots$ be random variables whose distribution is that of $Y+X_1+\cdots+X_{d-1}$. In addition, let $m$ be a $\mbox{Poi}(d)$ distributed random variable. We use the distributional equations derived by Lemma \ref{lem:aux_ind} to compute the expectation of $X$.
\begin{align}
\E[X]~=~&\E\left[\frac{\textbf{1}_{\{P>0,~\forall j\in [m]:\;Q_j>0\}}}{ 1+P^{-1}+  \sum_{j=1}^{m}Q_j^{-1} }\right]\nonumber\\
=~&\E\left[{\textbf{1}_{\{P>0,~\forall j\in [m]:\;Q_j>0\}}}\left(1- \frac{P^{-1}+  \sum_{j=1}^{m}Q_j^{-1}}{ 1+  P^{-1}+  \sum_{j=1}^{m}Q_j^{-1}}\right)\right]\nonumber\\
=~&a-
\E\left[\frac{P^{-1}\cdot \textbf{1}_{\{P>0,~\forall j\in [m]:\;Q_j>0\}}}{ 1+  P^{-1}+  \sum_{j=1}^{m}Q_j^{-1}}\right]
-\E\left[\sum_{i=1}^{m}\frac{ Q_i^{-1}\cdot \textbf{1}_{\{P>0,~\forall j:\;Q_j>0\}}}{ 1+  P^{-1}+ Q_i^{-1} +\sum_{\substack{j=1\\ j\neq i}}^{m}Q_j^{-1}}\right]\label{eqn:Eboth}
\end{align}
Let us separately expand the two expectations.
\begin{align}
\label{eqn:E1_1}
\E\left[\frac{P^{-1}\cdot \textbf{1}_{\{P>0,~\forall j\in [m]:\;Q_j>0\}}}{ 1+  P^{-1}+  \sum_{j=1}^{m}Q_j^{-1}}\right]
~=~&
\E\left[\frac{\frac{1}{1+\sum_{j=1}^{m}Q_j^{-1}}}{\frac{1}{1+\sum_{j=1}^{m}Q_j^{-1}}+P}\cdot \textbf{1}_{\{P>0,~\forall j\in [m]:\;Q_j>0\}}\right]\\
=~&\E\left[\frac{Y\cdot\textbf{1}_{\{X_1+\cdots+X_d>0,\; Y>0\}}}{Y+X_1+\cdots+X_d}\right].
\label{eqn:E1_2}
\end{align}
For the equality (\ref{eqn:E1_1}) we multiply both the nominator and the denominator by $P/(1+\sum_{j=1}^{m}Q_j^{-1})$. The distributional equation of $Y$ yields Equation (\ref{eqn:E1_2}). 

The second expectation in (\ref{eqn:Eboth}) requires a little more work.
\begin{align}
\label{eqn:E2_1}
\E\left[\sum_{i=1}^{m}\frac{ Q_i^{-1}\cdot \textbf{1}_{\{P>0,~\forall j:\;Q_j>0\}}}{ 1+  P^{-1}+ Q_i^{-1} +\sum_{\substack{j=1\\ j\neq i}}^{m}Q_j^{-1}}\right]
~=~&\E_m\left[m\cdot\E\left[\frac{ Q^{-1}\cdot \textbf{1}_{\{P>0,~Q>0~,~\forall j:\;Q_j>0\}}}{ 1+  P^{-1}+Q^{-1} + \sum_{j=1}^{m-1}Q_j^{-1}}\right]\right]\\
\label{eqn:E2_2}
~=~& d\cdot\E\left[\frac{ Q^{-1}\cdot \textbf{1}_{\{P>0,~Q>0~,~\forall j:\;Q_j>0\}}}
{ 1+  P^{-1}+Q^{-1} + \sum_{j=1}^{m}Q_j^{-1}}\right]\\
\label{eqn:E2_3}
~=~& d\cdot\E\left[
\frac{ \frac{1}{ 1+  P^{-1}+ \sum_{j=1}^{m}Q_j^{-1}} }{ Q+\frac{1}{ 1+  P^{-1}+ \sum_{j=1}^{m}Q_j^{-1}}}
\cdot \textbf{1}_{\{P>0,~Q>0~,~\forall j:\;Q_j>0\}}
\right]\\
\label{eqn:E2_4}
~=~&d\cdot \E\left[\frac{X\cdot\textbf{1}_{\{Y+X_1+\cdots+X_{d-1}>0,\; X>0\}}}{Y+X_1+\cdots+X_{d-1}+X}\right]\\
\label{eqn:E2_5}
~=~&\E\left[\sum_{i=1}^{d}\frac{X_i \cdot\textbf{1}_{\{Y+X_1+\cdots+X_{i-1}+X_{i+1}+\cdots+X_{d}>0,\; X_i>0\}}}{Y+X_1+\cdots+X_{d}}\right].
\end{align}
Equation (\ref{eqn:E2_1}) follows by linearity of expectation and the symmetry of the $Q_i$'s. To derive (\ref{eqn:E2_2}) we note that for every function $\varphi:\N\to\R,~~\E_m[m\cdot\varphi(m-1)]=d\cdot\E_m[\varphi(m)]$ where $m$ is Poi$(d)$ distributed. Equations (\ref{eqn:E2_3}) and (\ref{eqn:E2_4}) are obtained similarly to (\ref{eqn:E1_1}) and (\ref{eqn:E1_2}). The last equation (\ref{eqn:E2_5}) is derived by linearity of expectation and the symmetry of the $X_i$'s.

Let us return to the main computation of $\E[X]$ by plugging in (\ref{eqn:E1_2}) and (\ref{eqn:E2_5}) into (\ref{eqn:Eboth}).
\begin{align}
\nonumber
\E[X]~=~&
a-
\E\left[ 
\frac{Y\cdot\textbf{1}_{\{X_1+\cdots+X_d>0,\; Y>0\}}+
\sum_{i=1}^{d} X_i \cdot\textbf{1}_{\{Y+X_1+\cdots+X_{i-1}+X_{i+1}+\cdots+X_{d}>0,\; X_i>0\}}
}
{Y+X_1+\cdots+X_d}
\right]\\
~=~& a-b(1-(1-a)^d)-(1-b)(1-(1-a)^d-da(1-a)^{d-1}).
\label{eqn:final}
\end{align}
Equation (\ref{eqn:final}) is derived by the following observation. To compute the expectation in the preceding line we sample $Y,X_1,...,X_d$ independently. If two or more of these variables are positive, the contribution to the expectation is $1$. Otherwise the contribution is $0$. Therefore, the expectation equals the probability that two or more of these variables are non-zero. 

To conclude the proof, we recall that $a=b(1-(1-a)^d)$ by (\ref{eqn:ab}), hence $E[X]\le 0.$
\end{proof}
\subsection{The random $(1-\varepsilon)$-out  $d$-complex }
We turn to prove Theorem \ref{thm:1eps}. Let us denote by $Z(n,m,d)$ the number of $n$-vertex inclusion-minimal $d$-cycles whose number of $d$-faces is $m$. In (\cite{ALLM}, Theorem 4.1), it is proved that there exists some constant $\delta=\delta(d)>0$ such that 
\[
\sum_{m=d+2}^{\delta n^d} Z(n,m,d)\left(\frac{d+1}{n}\right)^m = o(1).
\]
The probability of any complex with $m$ $d$-faces to be a subcomplex of $S$ is at most $\left(\frac{d+1}{n}\right)^m$ since the $d$-faces of $S=S_d(n,1)$ are non-positively correlated. Therefore, a.a.s., the $d$-cycles in $S$ are either of size at least $\delta n^d$ (=large) or are boundaries of a $(d+1)$-dimensional simplex.

Fix some $\varepsilon>0$, and let  $S'=S_d(n,1-\varepsilon)$. Recall that $S'$ is obtained from $S$ by removing every $d$-face independently with probability $\varepsilon$. In order to analyze the $d$-homology of $S'$, we remove the $d$-faces of $S$ sequentially. I.e., we consider a sequence of complexes $S=S_0\supset S_1 \supset ... \supset S_M=S'$, where a.a.s.\ $M=\Theta(n^d)$. In every step $i$ of this process, $\beta_d(S_i)<\beta_d(S_{i-1})$ if and only if the $d$-face we removed participated in a $d$-cycle of $S_i$. If $S_i$ contains a large $d$-cycle, that is not a boundary of a $(d+1)$-simplex, the Betti number decreases with probability bounded away from zero. Since $\beta_d(S_0)=o(n^d)$ and $M=\theta(n^d)$, the probability that any large $d$-cycle survives the process is negligible.

\section{Discussion and open questions}
\label{sec:open}
The study of hypertrees raises many open questions. Here are some that concern enumeration and randomized constructions.

\begin{itemize}
\item There is no reason to believe that the bounds in Theorem \ref{thm:1} are tight. In fact, it is known \cite{LAB} that a similar argument for $d=1$ does not yield a bound of $((1-o(1)n)^n$. However, as mentioned in the introduction, it is conceivable that Conjecture \ref{conj:main} can be answered using Theorem \ref{thm:1} coupled with non trivial upper bounds for the number of $d$-collapsible hypertrees. It is also interesting to improve our lower bound for $|\Cnd|$ which at the moment is supported on non-evasive complexes that have a very irregular vertex-degree sequence (See \cite{nonevase}).

\item Is it possible to efficiently sample $d$-hypertrees uniformly at random? It is suggestive to do this using rapidly mixing Markov chains (e.g., \cite{sinclair}), possibly the {\em base-exchange} Markov Chain $\Omega=\Omega_{n,d}$ that is of interest for matroids in general \cite{balanced}. The states of $\Omega$ are all the $n$-vertex $d$-hypertrees. To proceed from a $d$-hypertree $T$, we select a $d$-face $\sigma\notin T$ uniformly at random, and replace $T$ by $T\setminus\{\tau\}\cup\{\sigma\}$, where  $\tau$ is a random $d$-face in the unique $d$-cycle of $T\cup\{\sigma\}$. The stationary distribution of $\Omega$ is uniform, but we do not know whether it is rapidly mixing. 

\item Can the random $1$-out complex help us improve our estimates for the number of $d$-hypertrees? The problem boils down to bounding the typical {\em permanent} of such a $d$-complex $T$. Namely, the number of injective functions from $F_d(T)$ to $F_{d-1}(T)$ that map  every $d$-face to one of its subfaces. In other words, the permanent of $T$ is the number of maximum matchings in $T$'s inclusion graph. We wonder if the typical permanent of $S_d(n,1)$ can be bounded in terms of its local weak limit $\textbf{B}_d$ (See \cite{mik}).  

\item We know even less about random generation of $d$-collapsible hypertrees. It is possible to restrict the base-exchange chain $\Omega_{n,d}$ to $d$-collapsible hypertrees, but we do not even know whether the restricted chain is connected, not to speak of rapid mixing. There is also an interesting greedy-random process that suggests itself, where we sequentially add a random $d$-face to the current complex provided that $d$-collapsibility is not violated. How many $d$-faces does this process acquire before it halts? What is the combinatorial structure of the final complex?

\item Other types of hypertrees such as contractible, $\Z$-hypertrees, and $\F_2$-hypertrees can be considered in all these contexts. For instance, one can ask whether Theorem \ref{thm:1out} also holds over $\F_2$ coefficients. In particular, applying a first moment method on the $\F_2$-cohomology of $S_2(n,1)$ yields the following interesting question. For  a fixed graph $G$ and a pair of vertices $i,j$, let $$\varphi_{i,j}(G) := \frac{1}{n-2}\left|\{k\notin\{i,j\}~:~i,j,k\mbox{~span an even number of edges in }G\}\right|.$$ In words, $\varphi_{i,j}$ is the probability that the selection of the edge $ij$ in $S_2(n,1)$ does not exclude $G$ from being a cocycle of the complex. Prove that
\[
\sum_{\mbox{$G$ an $n$-vertex graph}}
\left[\prod_{i,j}\varphi_{i,j} \right]=2^{o(n^2)}
\]
It is conceivable that this sum is of order $2^{\Theta(n)}$.
\end{itemize}

\noindent{\bf Acknowledgement.}\, The authors would like to thank Gil Kalai for many useful discussions and for suggesting the inductive construction leading to Equation (\ref{eqn:cnd}).

\end{document}